\documentclass[preprint]{elsarticle}
\usepackage{amsmath}
\usepackage{amsfonts}
\usepackage{amssymb}
\usepackage{amsthm}
\usepackage[pdftex,unicode]{hyperref}

\def\R{\mathbb R} 

\def\om{\omega}
\def\be{\operatorname{\beta}}

\def\al{\alpha}
\def\ga{\gamma}

\def\cF{\Phi}
\def\ph{\varphi}

\def\cl#1{\overline{#1}}
\def\clx#1#2{\overline{#2}^{#1}}
\def\int#1{\operatorname{int} (#1)}
\def\intx#1#2{\operatorname{int}_{#1} (#2)}

\def\sset#1{\{#1\}}
\def\set#1{\bbset#1\eeset}
\def\bbset#1:#2\eeset{\{#1\,:\,#2\}}

\def\bbsett#1:#2\eesett{\{#1\,:\,\text{#2}\}}

\def\la{\Lambda}

\newtheorem{assertion}{Assertion}
\newtheorem{proposition}{Proposition}
\newtheorem{theorem}{Theorem}

\newtheorem*{lemma*}{Lemma}
\newtheorem{cor}{Corollary}

\theoremstyle{definition}
\newtheorem{definition}{Definition}

\begin{document}
\begin{frontmatter}

\title {Functions on products of pseudocompact spaces}

\author{Evgenii Reznichenko} 
\ead{erezn@inbox.ru}

\address{Department of General Topology and Geometry, Mechanics and  Mathematics Faculty, 
M.~V.~Lomonosov Moscow State University, Leninskie Gory 1, Moscow, 199991 Russia}

\begin{abstract}
For every natural number $ n $, any continuous function on the product of
$ X_1 \times X_2 \times ... \times X_n $
pseudocompact spaces extends to a separately continuous function on the product
$ \be X_1 \times \be X_2 \times ... \times \be X_n $ of their Stone-Cech compactifications.
\end{abstract}

\end{frontmatter}

\section {Introduction}

In \cite {rezn94} it was proved that any continuous function $ f: X_1 \times X_2 \to \R $,
defined on the product of pseudocompact spaces $ X_1 $ and $ X_2 $, extends to a
separately continuous function $ \hat f: \be X_1 \times \be X_2 \to \R $
defined on the product of the Stone-Cech compactifications $ \be X_1 $ and $ \be X_2 $.
In \cite {reznusp98} this result was extended to three factors:
any continuous function $ f: X_1 \times X_2 \times X_3 \to \R $,
defined on the product of pseudocompact spaces $ X_1 $, $ X_2 $, $ X_3 $, extends to
separately continuous function $ \hat f: \be X_1 \times \be X_2 \times \be X_3 \to \R $
defined on the product of Stone-Cech compactifications $ \be X_1 $, $ \be X_2 $, $ \be X_3 $.
The question was also raised of whether this result can be extended to an arbitrary finite number of factors. This paper gives a positive answer to this question: for every natural number $ n $, any continuous function on the product $ X_1 \times X_2 \times ... \times X_n $ of pseudocompact spaces can be extended to a separately continuous function on
$ \be X_1 \times \be X_2 \times ... \times \be X_n $; see  Theorem~\ref {t3}.

Recall the Gliksberg theorem \cite {gli1959}: given pseudocompact
spaces $ X $ and $ Y $, any continuous function can be extended to a continuous function defined on $ \be X \times \be Y $ if and only if $ X \times Y$ is a pseudocompact.
As shown in \cite {rezn94}, the Grothendieck theorem \cite {grot1952}
 (for every countably compact $ X $ any countably compact subspace $ Y $ of $ C_p (X) $ is compact) is, in fact, equivalent to the statement:
any separately continuous function defined on the product $ X \times Y $ of countably compact spaces can be extended to a separately
continuous function defined on $ \be X \times \be Y $.
The problem of extending functions from products of spaces to Stone-Cech compactifications was considered in detail in the case of two factors in \cite {rezn94} and in the case of a larger number of factors in \cite {reznusp98}.
We collect information about the extension of functions from a product of spaces in the following theorem.
\begin {theorem} \label {ti}
Given spaces  $ X_1 $, $ X_2 $, ..., $ X_n $, let $ X = \prod_ {i = 1} ^ n X_i $ and $ \hat X = \prod_ {i = 1} ^ n \be X_i $ be topological products, and let
$ \cF: X \to \R $ be a function.
\begin {itemize}
\item [\rm (a)] If $ X $ is pseudocompact and $ \cF $ is continuous, then $ \cF $ extends to a continuous function $ \hat \cF: \hat X \to \R $.
\item [\rm (b)] If the spaces $ X_1 $, $ X_2 $, ..., $ X_n $ are pseudocompact and $ \cF $ is continuous, then $ \cF $ extends to a separately continuous function $ \hat \cF: \hat X \to \R $.
\item [\rm (c)] If the spaces $ X_1 $, $ X_2 $, ..., $ X_n $ are countably compact and $ \cF $ is separately continuous, then $ \cF $ extends to a separately continuous function $ \hat \cF: \hat X \to \R $.
\end {itemize}
\end {theorem}
\begin {proof}
(a) For $ n = 2 $, this is exactly the Gliksberg theorem. For $ n> 2 $, the statement follows from Gliksberg theorem and the fact that the product of a compact space and a pseudocompact space is a pseudocompact space.

(b) This is Theorem \ref {t3}, the main result of this paper.

(c) For $ n = 2 $, this is exactly the Grothendieck theorem (in the product formulation). For $ n> 2 $, the assertion follows from Grothendieck theorem and the fact that the product of a compact space and a countably compact space is a countably compact space (Theorem 3.16 of \cite {rezn94}).
\end {proof}

Assertion (b) of Theorem \ref {ti} cannot be strengthened even for $ n = 2 $; see~\cite {rezn94}.
\begin {itemize}
\item [(a)] {\it There are countably compact spaces $ X_1 $ and $ X_2 $ and a continuous function $\cF: X_1\times X_2\to \R$ such that $ \cF $ does not extend to a continuous function on $ \be X_1 \times \be X_2 $.}
It suffices to take countably compact $ X_1 $ and $ X_2 $ such that $ X_1 \times X_2 $ are not pseudocompact (see \cite{Novak1953,rezn2020}) and use Gliksberg theorem to find $\cF: X_1\times X_2\to \R$ such that $ \cF $ does not extend to a  continuous function on $ \be X_1 \times \be X_2$.
\item [(b)] {\it There are pseudocompact spaces $ X_1 $ and $ X_2 $ and a separately continuous function $\cF: X_1\times X_2\to \R$ such that $ \cF $ does not extend to a separately continuous function on $ \be X_1 \times \be X_2 $.}
Indeed, it suffices to take a pseudocompact Tychonoff space $ X_1 $  all of whose countable subsets of which are closed and $C^*$-embedded; (see \cite {sha1986}), put $ X_2 = C_p (X_1, [0,1])$, and define $ \cF $: $ \cF (x, f) = f (x) $ for $ x \in X_1 $ and $ f \in X_2 $.
\end {itemize}

\section {Notation and preliminaries}

All spaces are assumed to be Tychonoff.
The set of nonnegative integers is denoted by $ \om $.
If $ X $ is a space and $ M \subset X $, then we denote by $ \clx XM $ the closure of $ M $
and $ \intx XM $ is the interior of $ M $. If it is clear from the context 
what space $ X $ we are talking about, then we write $ \cl M $ and $ \int M $
instead of $ \clx XM $ and $ \intx XM $.

Let $ A $ be a set and let  $ \set {X_ \al: \al \in A} $ a family of spaces. For $ B \subset A $, $ X = \prod _ {\al \in A} X_ \al $ and 
$ y \in \prod _ {\al \in A \setminus B} X_ \al $,
we denote by $ \pi ^ X_B: X \to \prod _ {\al \in B} X_ \al $ the natural projection. For $ f \in \R ^ X $, we
denote by $ r (f, X, y) $ the function $ g: \prod _ {\al \in B} X_ \al \to \R $ such that $ g (x) = f (z) $,
where $ z \in X $, $ \pi_ {A \setminus B} (z) = y $, and $ x = \pi_B ^ X (z) $. Note that
$$
r (f, X, \pi ^ X_ {A \setminus B} (z)) (\pi ^ X_B (z)) = f (z).
$$
We denote by $ p (X, y) $ the mapping $ h: \prod _ {\al \in B} X_ \al \to X $ defined by $ h (x) = z $,
where $ z \in X $, $ \pi_ {A \setminus B} (z) = y $, and $ x = \pi_B ^ X (z) $.
Note that
\begin {gather*}
p (X, \pi ^ X_ {A \setminus B} (z)) (\pi ^ X_B (z)) = z,
\\
r (f, X, y) (x) = f (p (X, y) (x)).
\end {gather*}
We denote by $ \la_B ^ X $ the natural bijection
\[
\la_B ^ X: \R ^ X \to (\R ^ {\prod _ {\al \in A \setminus B} X_ \al}) ^ {\prod _ {\al \in B} X_ \al}.
\]
The map $ \la_B ^ X $ is a homeomorphism.

For  spaces $ Y $ and $ Z $, we denote by $ C (Y, Z) $ the set of continuous functions from $ Y $ to $ Z $; we put
$ C (Y) = C (Y, \R) $.
We denote by $ SC (X, Z) $ separately continuous functions from $ X $ to $ Z $,
$ SC (X) = SC (X, \R) $.

Denote by $ C_p (Y, Z) $, $ C_p (Y) $, $ SC_p (X, Z) $ and $ SC_p (X) $ the function
spaces  $ C (Y, Z) $, $ C (Y) $, $ SC (X, Z) $ and $ SC (X) $, respectively,
in the topology of pointwise convergence.

Note that
\begin {gather*}
SC_p (X) \subset \R ^ X,
\\
SC_p (\prod _ {\al \in B} X_ \al, SC_p (\prod _ {\al \in A \setminus B} X_ \al)) \subset
(\R ^ {\prod _ {\al \in A \setminus B} X_ \al}) ^ {\prod _ {\al \in B} X_ \al}.
\end {gather*}
The map $ \la_B ^ X $ is a homeomorphism of $ SC_p (X) $ onto
\[
SC_p (\prod _ {\al \in B} X_ \al, SC_p (\prod _ {\al \in A \setminus B} X_ \al)).
\]

Using Lemma 3.7~of \cite {reznusp98}, we reformulate Proposition 3.1 of \cite {reznusp98} as follows (Assertion 1.4 \cite {rezn94}).

\begin {proposition}[Proposition 3.1 \cite {reznusp98}, Assertion 1.4 \cite {rezn94}] \label {p1}
Suppose that $ X_1 $ and $ X_2 $ are pseudocompact spaces, $ X = X_1 \times X_2 $, $ \cF \in SC (X) $ and 
$ \ph = \la _ {\sset 1} ^ X (\cF): X_1 \to C_p (X_2) $.
Then the following conditions are equivalent:
\begin {itemize}
\item [\rm(a)] $ \cF $ can be extended to a separately continuous function on $ X_1 \times \be X_2 $;
\item [\rm(b)] $ \cF $ can be extended to a separately continuous function on $ \be X_1 \times X_2 $;
\item [\rm(c)] $ \cF $ can be extended to a separately continuous function on $ \be X_1 \times \be X_2 $;
\item [\rm(d)] the closure of $ \ph (X_1) $ in $ C_p (X_2) $ is compact;
\item [\rm(e)] $ \ph (X_1) $ is compact;
\item [\rm(f)] $ \ph (X_1) $ is an Eberlein compact space.
\end {itemize}
\end {proposition}

Theorem 3.15 and Claim 3.6 of \cite {reznusp98} imply the following proposition.

\begin {proposition}
Let $ X_1, X_2, ..., X_ {n-1} $ be countably compact spaces, let $ X_n $ be a pseudocompact space and let $ X = \prod_ {i = 1} ^ n X_i $.
Then any $ \cF \in SC (X) $ can be extended to a separately continuous function on $ \prod_ {i = 1} ^ n \be X_i $.
\end {proposition}

\begin {cor} \label {c1}
Let $ X_1$ and $X_2 $ be countably compact spaces, let $ X_3 $ be a pseudocompact space and let $ X = X_1 \times X_2 \times X_3 $.
Then any $ \cF \in SC (X) $ can be extended to a separately continuous function on $ \be X_1 \times \be X_2 \times \be X_3 $.
\end {cor}

\begin {definition}
Let $ A $ be a set, let $ \set {X_ \al: \al \in A} $ be a family of spaces, and let $ X = \prod _ {\al \in A} X_ \al $.
Suppose given a function $ f: X \to \R $ and a positive integer  $ n $.
\begin {itemize}
\item (Definition 3.25 of \cite {reznusp98})
 The function $ f $ is {\it $ n $-separately continuous} iff $ r (f, X, y): \prod _ {\al \in B} X_ \al \to \R $  is continuous for each $ B \subset A $ with $ | B | \le n $ and any $ y \in \prod _ {\al \in A \setminus B} X_ \al $.
\item The function $ f $ is {\it $ n $-$ \be $-extendable} iff $ g = r (f, X, y): \prod _ {\al \in B} X_ \al \to \R $ extends to a separately continuous function $ \hat g: \prod _ {\al \in B} \be X_ \al \to \R $ for each $ B \subset A $ whith $ | B | \le n $ and any $ y \in \prod _ {\al \in A \setminus B} X_ \al $.
\end {itemize}
\end {definition}

Separately continuous functions are exactly $ 1 $-separately continuous functions.

\begin {proposition}[Proposition 1.12 of \cite {rezn94}; see also Theorem 3.27 of \cite {reznusp98}] \label {ecf2} 
Let $ X_1 $, $ X_2 $ be a pseudocompact space, $ \cF: X_1 \times X_2 \to \R $ a continuous function.
Then $ \cF $ extended to a separately continuous function on $ \be X_1 \times \be X_2 $.
\end {proposition}

\section {Auxiliary Results}

Proposition \ref {ecf2} implies the following statement.

\begin {proposition} \label {ecf3}
Let $ X_1, X_2, ..., X_ {n} $ be pseudocompact spaces and let $ X = \prod_ {i = 1} ^ n X_i $.
Then any $ 2 $-separately continuous function on $ X $ is $ 2 $-$ \be $-extendable.
\end {proposition}

\begin {assertion} \label {a1}
Let $ X_1 $, $ X_2 $, and $ X_3 $ be pseudocompact spaces and let $ \cF: X_1 \times X_2 \times X_3 \to \R $  be a $2$-$\be$-extendable function.
Then $ \cF $ extends to a separately continuous function $ \hat \cF: \be X_1 \times \be X_2 \times \be X_3 \to \R $.
\end {assertion}
\begin {proof}
We set $ Y_1 = \be X_1 $, $ Y_2 = \bigcup \set {\clx {\be X_2} M: M \subset X_2, \, | X | \le \om} $ and $ Y_3 = X_3 $.
\begin {lemma*}
The function $\cF$ extends to separately continuous function $ \Psi: Y_1 \times Y_2 \times Y_3 \to \R $.
\end {lemma*}
\begin {proof}
Let $ z \in Y_3 $.
Then $ z \in X_3 $ and the function $ \cF $ $ 2 $-$ \be $-extendable; therefore, the function
\[
g_ {z} = r (\cF, X_1 \times X_2 \times X_3, z): X_1 \times X_2 \to \R
\]
extends to a separately continuous function
\[
\hat g_ {z}: \be X_1 \times \be X_2 \to \R.
\]
Consider  the function $ \Psi: Y_1 \times Y_2 \times Y_3 \to \R $ defined by setting
\[
\Psi (x, y, z) = \hat g_ {z} (x, y)
\]
for $x\in Y_1$ and $y\in Y_2$.
Let us check that the function $ \Psi $ is separately continuous.
Given $ (y_1, y_2, y_3) \in Y_1 \times Y_2 \times Y_3 $ we must show that the functions
\begin {align*}
h_1 & = r (\Psi, Y_1 \times Y_2 \times Y_3, (y_2, y_3)): Y_1 \to \R,
\\
h_2 & = r (\Psi, Y_1 \times Y_2 \times Y_3, (y_1, y_3)): Y_2 \to \R,
\\
h_3 & = r (\Psi, Y_1 \times Y_2 \times Y_3, (y_1, y_2)): Y_3 \to \R
\end {align*}
are continuous. Since
\begin {align*}
h_1 (y_1') & = \Psi (y_1', y_2, y_3) = \hat g_ {y_3} (y_1', y_2),
\\
h_2 (y_2') & = \Psi (y_1, y_2', y_3) = \hat g_ {y_3} (y_1, y_2')
\end {align*}
for $ y_1'\in Y_1 $ and $ y_2' \in Y_2 $ and $ \hat g_ {y_3} $ is a separately continuous function,
it follows that the functions $ h_1 $ and $ h_2 $ are continuous.

It remains to check the continuity of the function $ h_3 $. There is a countable $ P \subset Y_2 $ such that $ y_2 \in \cl P $. Let $ P = \set {p_n: n \in \om} $. We put
\begin {align*}
f_n & = r (\cF, X_1 \times X_2 \times X_3, (p_n)): X_1 \times X_3 \to \R,
\\
f_n (x, y) & = \cF (x, p_n, y) \text {for} (x, y) \in X_1 \times X_3,
\\
\ph_n & = \la _ {\sset 1} ^ {X_1 \times X_3} (f_n): X_1 \to C_p (X_3).
\end {align*}
Since the function $ \cF$ is $2 $-$ \be $-extendable, it follows that $ f_n $
extends to a separately continuous function
$ \hat f_n: \be X_1 \times \be X_3 \to \R $.
Proposition \ref {p1} implies that $ \ph_n (X_1) $ is an Eberlein compactum.
Let 
\[
\hat \ph_n: \be X_1 \to \ph_n (X_1)
\]
be a continuous extension of the map $ \ph_n $.
We put
\begin {align*}
E & = \prod_ {n \in \om} \ph_n (X_1),
\\
\ph & = \mathop {\Delta} \limits_ {n \in \om} \ph_n: X_1 \to P.
\end {align*}
The space $E$ is a countable product of Eberlein compacta and therefore is an Eberlein compactum. The mapping $ \ph $ is continuous; hence $ \ph (X_1) $ is a pseudocompact subspace of an Eberlein compactum. Since the pseudocompact subspaces of Eberlein compacts are closed (see Corollary 6 of \cite {ps1974}), $ \ph (X_1) $ is an Eberlein compactum.
Let 
\[
\hat \ph: \be X_1 \to E
\]
be a continuous extension of the map $ \ph $. Since $ \ph (X_1) $ is compact, we have
$ \ph (X_1) = \hat \ph (\be X_1) $. Thus, there is a $ y_1'\in X_1 $ such that $ \ph (y_1') = \hat \ph (y_1) $. Note that
\[
\ph_n (y_1') = \hat \ph_n (y_1)
\]
for each $ \in \om $.
We put
\[
h_3'= r (\Psi, Y_1 \times Y_2 \times Y_3, (y_1', y_2)): Y_3 \to \R.
\]

Let us check that $ h_3 = h_3'$. Suppose that , on the contrary, $ h_3 \neq h_3'$.
Then $ h_3 (z) \neq h_3'(z) $ for some $ z \in X_3 $.
Since
\begin {align*}
h_3 (z) & = \Psi (y_1, y_2, z) = \hat g_z (y_1, y_2)
\\
h_3'(z) & = \Psi (y_1', y_2, z) = \hat g_z (y_1', y_2)
\end {align*}
in follows that 
\[
\hat g_z (y_1, y_2) \neq \hat g_z (y_1', y_2), 
\]
and since the function $ \hat g_z $ is separately continuous and $ y_2 \in \cl P $, it follows that 
\[
\hat g_z (y_1, p_n) \neq \hat g_z (y_1', p_n)
\]
for some $ n \in \om $.
Finally, since
\begin {align*}
\hat g_z (y_1, p_n) & = \hat \ph_n (y_1) (z),
\\
g_z (y_1', p_n) & = \ph_n (y_1') (z),
\end {align*}
it follows that $ \hat \ph_n (y_1) (z) \neq \ph_n (y_1') (z) $. This contradicts the fact that $ \ph_n (y_1') = \hat \ph_n (y_1) $
for each $ n \in \om $.

So, we have shown that $ h_3 = h_3'$. We put
\begin {align*}
q & = r (\cF, X_1 \times X_2 \times X_3, (y_1')): X_2 \times X_3 \to \R,
\\
q (x, y) & = \cF (y_1', x, y) \text { for } (x, y) \in X_2 \times X_3,
\\
\psi & = \la _ {\sset 2} ^ {X_2 \times X_3} (q): X_2 \to C_p (X_3).
\end {align*}
The function $ \cF $ is $ 2 $-$ \be $-extendable and hence the function $ q $
extends to a separately continuous function
$ \hat q: \be X_2 \times \be X_3 \to \R $.
Proposition \ref {p1} implies that $ \psi (X_2) $ is an Eberlein compactum.
Let 
\[
\hat \psi: \be X_2 \to \psi (X_2)
\]
be a continuous extension of the mapping $ \psi $. Since $ \psi (X_2) = \hat \psi (\be X_2) $,
we have $ \psi (y_2') = \hat \psi (y_2) $ for some $ y_2' \in X_2 $.
We put
\[
h_3'' = r (\Psi, Y_1 \times Y_2 \times Y_3, (y_1', y_2')): Y_3 \to \R.
\]
Let us check that $ h_3 = h_3'= h_3''$. Suppose that, on the contrary, $ h_3'\neq h_3''$.
Then $ h_3'(z) \neq h_3''(z) $ for some $ z \in X_3 $.
Since
\begin {align*}
h_3'(z) & = \Psi (y_1', y_2, z) = \hat g_z (y_1', y_2),
\\
h_3'' (z) & = \Psi (y_1', y_2, z) = \hat g_z (y_1', y_2'),
\end {align*}
we have
\[
\hat g_z (y_1', y_2) \neq \hat g_z (y_1', y_2'),
\]
and since
\begin {align*}
\hat g_z (y_1', y_2) & = \hat \psi (y_2) (z),
\\
\hat g_z (y_1', y_2') & = \psi (y_2') (z),
\end {align*}
we have $ \hat \psi (y_2) (z) \neq \psi (y_2') (z) $. This contradicts the fact that $ \psi (y_2') = \hat \psi (y_2) $.

Thus, we have shown that $ h_3 = h_3'= h_3''$. Since $ y_1'\in X_1 $, $ y_2' \in X_2 $, and $ X_3 = Y_3 $,
we have
\[
h_3'' = r (\cF, X_1 \times X_2 \times X_3, (y_1', y_2')): X_3 \to \R.
\]
Since the function $ \cF $ is separately continuous, it follows that the function $ h_3 = h_3'' $ is continuous.
\end {proof}
Let $ \Psi: Y_1 \times Y_2 \times Y_3 \to \R $ be a separately continuous extension of the function $ \cF $.
Since $ Y_1 $ is compact, $ Y_2 $ is countably compact and $ Y_3 $ is pseudocompact, it follows from the Corollary \ref {c1} that the function $ \Psi $ can be extended to a separately continuous function $ \hat \cF: \be Y_1 \times \be Y_2 \times \be Y_3 \to \R $. To prove the assertion, it remains to note that $ \be X_1 = \be Y_1 $, $ \be X_2 = \be Y_2 $, and $ \be X_3 = \be Y_3 $.
\end {proof}

In view of Proposition \ref {ecf3}, it follows that the Assertion \ref {a1} generalizes Theorem 3.27 of \cite {reznusp98}.

\begin {assertion} \label {a2}
Let $ A $ be a set and let $ X_ \al $ be a pseudocompact space for each $ \al \in A $. Then any $2$-$\be$-extendable function
\[
\cF: X = \prod _ {\al \in A} X_ \al \to \R
\]
is a $3$-$\be $-extendable function.
\end {assertion}
\begin {proof}
Let $ B \subset A $, $ | B | \leq 3 $, and let $ y \in \prod _ {\al \in A \setminus B} X_ \al $. The function
\[
f = r (X, \cF, y): \prod _ {\al \in B} X_ \al \to \R.
\]
is a $2$-$ \be $-extendable and $ | B | \leq 3 $; therefore, Assertion \ref {a1} implies that the function $ f $ can be extended to a separately continuous function
\[
\hat f: \prod _ {\al \in B} \be X_ \al \to \R.
\]
\end {proof}

\begin {assertion} \label {a3}
Suppose given a set $A$, spaces $ X_ \al $ for all $ \al \in A $, 
a positive integer $ m $, and a set
$ B \subset A $ with $ | B | <m $. Let
\[
Y_ \al =
\begin {cases}
X_ \al & \text{if } \al \notin B \\
\be X_ \al & \text{if } \al \in B \\
\end {cases}
\]
for each $ \al \in A $ and let
\[
X = \prod _ {\al \in A} X_ \al,
\qquad
Y = \prod _ {\al \in A} Y_ \al.
\]
Then any $ m $-$ \be $-extendable 
function $ \cF: X \to \R $ extends to an $ l $-$ \be $-extendable 
function $ \Psi: Y \to \R $, where $l = m - |B|$.
\end {assertion}
\begin {proof}
Let us define the function $ \Psi $. For $ y \in Y $, we set $ u = \pi_B (y) $,
\[
v = \pi_ {A \setminus B} (y) \in \prod _ {\al \in A \setminus B} X_ \al,
\qquad
f = r (\cF, X, v): \prod _ {\al \in B} X_ \al \to \R.
\]
Since $ \cF $ is $ m $-$ \be $-extendable, the function $ f $ extends to a separately continuous function $ \hat f: \prod _ {\al \in B} \be X_ \al \to \R $.  We define the function $\Psi$ by setting 
 $ \Psi (y) = \hat f (u) $. 

Let us check that $ \Psi $ is an $ l $-$ \be $-extendable function. Let $ C \subset A $, $ | C | \leq l $.
For $ z \in \prod _ {\al \in A \setminus C} Y_ \al $ consider the function
\[
 g = r (\Psi, Y, z): \prod _ {\al \in C} Y_ \al \to \R.
\]
It is enough to prove that the function $ g $ extends to a separately continuous function $ q: \prod _ {\al \in C} \be Y_ \al \to \R $. Note that $ \be X_ \al = \be Y_ \al $ for $ \al \in A $, so that $ \prod _ {\al \in C} \be Y_ \al = \prod _ {\al \in C} \be X_ \al $. We set $ D = B \cup C $ and $ y '= \pi_ {A \setminus D} (y) $ and consider
\[
 h = r (\cF, X, y '): \prod _ {\al \in D} X_ \al \to \R.
\]
Since $ | D | \leq m $ and $ \cF $  is an $m$-$\be$-extendable function, it follows that  $ h $ can be extended to a separately continuous function
$ \hat h: \prod _ {\al \in D} \be X_ \al \to \R $. We put
\[
q = r (\prod _ {\al \in D} \be X_ \al, \hat h, \pi_ {D \setminus C} (y)): \prod _ {\al \in C} \be X_ \al \to \R.
\]
Since $ \hat h $ is a separately continuous function, $ q $ is separately continuous as well. By construction, $ q $ extends $ g $.
\end {proof}

\begin {cor} \label {c2}
Suppose given a set $A$, spaces $ X_ \al $ is the space for $ \al \in A $, and $ \ga \in A $. Let
\[
Y_ \al =
\begin {cases}
X_ \al &\text{if } \al \neq \ga, \\
\be X_ \al &\text{if } \al = \ga \\
\end {cases}
\]
for  each $ \al \in A $ and let
\[
X = \prod _ {\al \in A} X_ \al,
\qquad
Y = \prod _ {\al \in A} Y_ \al.
\]
Then any $ 3 $-$ \be $-extendable 
function $ \cF: X \to \R $ can be extended to a $ 2 $-$ \be $-extendable function $ \Psi: Y \to \R $.
\end {cor}

Assertion \ref {a2} and Corollary \ref {c2} imply the following assertion.

\begin {assertion} \label {a4}
Suppose given a set $A$, pseudocompact spaces 
$ X_ \al $ for all  $ \al \in A $ and $ \ga \in A $,
\[
Y_ \al =
\begin {cases}
X_ \al &\text{if } \al \neq \ga, \\
\be X_ \al &\text{if } \al = \ga \\
\end {cases}
\]
for each $ \al \in A $ and let
\[
X = \prod _ {\al \in A} X_ \al,
\qquad
Y = \prod _ {\al \in A} Y_ \al.
\]
Then any $ 2 $-$ \be $-extendable
function $ \cF: X \to \R $ can be extended to a $2$-$\be$-extendable function $ \Psi: Y \to \R $.
\end {assertion}

\section {Main results}

\begin {theorem} \label {t1}
Let $ X_1 $, $ X_2 $, ..., $ X_n $ be pseudocompact spaces and let $ \cF: \prod_ {i = 1} ^ n X_i \to \R $  be a $2$-$\be$-extendable function.
Then $ \cF $ extends to a separately continuous function $ \hat \cF: \prod_ {i = 1} ^ n \be X_i \to \R $.
\end {theorem}
\begin {proof}
If $ n \leq 3 $, then the theorem follows from Assertion \ref {a1}. In what follows, we will assume that $ n> 3 $. We put
\[
Y_ {l, m} =
\begin {cases}
\be X_m &\text{if }  m \leq l, \\
 X_m &\text{if }  m> l \\
\end {cases}
\]
for $ l = 0,1, ..., n $ and $ m = 1, ..., n $ and
$
Y_l = \prod_ {m = 1} ^ n Y_ {l, m}
$
for $ l = 0,1, ..., n $. Note that
\[
Y_0 = \prod_ {i = 1} ^ n X_i
\text { and }
Y_n = \prod_ {i = 1} ^ n \be X_i.
\]
By induction on $ l $ we define 2-$\be$-extendable mappings $ \Psi_l: Y_l \to \R $ for $ l = 0,1, ..., n $ in such a way that $ \Psi_ {l} $ extends $ \Psi_ {l-1} $ for $ l> 0 $ and $ \Psi_0 = \cF $.

For $ l = 0 $, we put $ \Psi_0 = \cF $.

Suppose that  $ l> 0 $ and a 2-$\be$-extendable mapping $ \Psi_ {l-1}: Y_ {l-1} \to \R $ is constructed. Each $ Y_ {l, m} $ is pseudocompact,
$ Y_ {l, m} = Y_ {l-1, m} $ for $ l \neq m $, and $ Y_ {m, m} = \be Y_ {m-1, m} = \be X_m $. Assertion \ref {a4} implies that the function $ \Psi_ {l-1} $ can be extended to a 2-$\be$-extendable mapping $ \Psi_l: Y_l \to \R $. This completes the inductive construction.

By construction, each $ \Psi_l $ is separately continuous and extends $ \cF $. Therefore, the mapping
\[
\hat \cF = \Psi_n: \prod_ {i = 1} ^ n \be X_i \to \R
\]
is separately continuous and extends $\cF$.
\end {proof}

Theorem \ref {t1} and Proposition \ref {ecf3} imply the following theorem.

\begin {theorem} \label {t2}
Let $ X_1 $, $ X_2 $, ..., $ X_n $ be pseudocompact spaces, and let $ \cF: \prod_ {i = 1} ^ n X_i \to \R $ be a $2$-separately continuous function.
Then $ \cF $ extends to a separately continuous function $ \hat \cF: \prod_ {i = 1} ^ n \be X_i \to \R $.
\end {theorem}

A continuous function on a product spaces is $2$-separately continuous, so Theorem \ref {t2} implies the following result.

\begin {theorem} \label {t3}
Let $ X_1 $, $ X_2 $, ..., $ X_n $ be pseudocompact spaces and let $ \cF: \prod_ {i = 1} ^ n X_i \to \R $ be a continuous function.
Then $ \cF $ extends to a separately continuous function $ \hat \cF: \prod_ {i = 1} ^ n \be X_i \to \R $.
\end {theorem}

\bibliographystyle{elsarticle-num}
\bibliography{scfps}

\end{document}